\documentclass[12pt]{article}

\usepackage[english]{babel}
\usepackage[utf8]{inputenc} 		
\usepackage{amssymb}
\usepackage{amsmath}
\usepackage{tikz-cd}				
\usetikzlibrary{matrix,arrows,decorations.pathmorphing}
\usepackage[all]{xy}				
\usepackage{hyperref}  			

\usepackage{mathrsfs}										
\usepackage{multicol}				
\usepackage{booktabs}			

\usepackage{anysize}

\marginsize{2cm}{3cm}{2.cm}{2.0cm}

\newtheorem{theorem}{Theorem}[section] 		

\newtheorem{example}[theorem]{Example}

\newtheorem{definition}[theorem]{Definition}

\newtheorem{corollary}[theorem]{Corollary}
\newenvironment{proof}{\paragraph{Proof:}}{\hfill$\square$ \\}

\newtheorem{remark}[theorem]{Remark}

\newtheorem{thm}{Theorem}

\DeclareMathOperator{\Ker}{Ker}
\DeclareMathOperator{\Img}{Im}

\DeclareMathOperator{\id}{id}

\DeclareMathOperator{\Span}{Span}

\newcommand{\Z}{{\mathbb{Z}}}
\newcommand{\Q}{{\mathbb{Q}}}

\newcommand{\K}{{\mathbb{K}}}

\begin{document}

\title{$A_\infty$ structures and Massey products}
\author{Urtzi Buijs, José M. Moreno-Fernández, Aniceto Murillo\footnote{The authors have been partially supported by the MICINN grant MTM2016-78647-P, and by the Junta de Andalucía grant FQM-213. \vskip 1pt 2010 Mathematics subject
classification: 55S30, 16E45, 18G55, 55P62, 55P60.\vskip
 1pt
 Key words and phrases: Higher Massey products. $A_\infty$-algebra. Rational homotopy theory.}}
\date{}
\maketitle

{\begin{center}\small \em In memory of W. S. Massey\end{center}}

\abstract{We show how and when it is possible to detect and recover higher Massey products on the cohomology $H$ of a differential graded algebra $A$ with higher multiplications on quasi-isomorphic $A_\infty$ structures on $H$.}\\

\section{Introduction}

Higher order Massey products were introduced in \cite{Massey} and generalized in \cite{MayMatric}. These are of fundamental importance not only in the study of differential graded algebras (DGA's, henceforward) per se, but also  in those geometrical contexts where DGA's play a role. Classical instances of this fact are the detection of linking numbers of knots \cite{massey1998higher} or the obstructions to formality of Kähler manifolds \cite{deligne1975real}. Recently, Massey products have proved to be useful in a wide range of applications which go from symplectic geometry \cite{babenko2000massey} to algebraic geometry \cite{taylor2010controlling} passing through homotopy \cite{Longoni2005375}, group  \cite{ruiz2013cohomological}, and number theory \cite{minavc2017triple}.

On the other hand, and since their introduction   in relation with the homotopy theory of $H$-spaces \cite{stasheff1,stasheff2}, $A_\infty$ algebras have also been successfully used in geometry, algebra and mathematical physics, see for instance \cite{Merkulov}, \cite{lu2009infinity} and \cite{Fukaya} respectively. A particulary pedagogical and complete introduction to these structures can be found in \cite{Keller}. An essential result  states that given any DGA $A$, there is a structure of minimal $A_\infty$ algebra on its cohomology $H$, unique up to $A_\infty$ isomorphism, for which $A$ and $H$ are quasi-isomorphic $A_\infty$ algebras {\cite{Kadeishvili}}. However, this unique isomorphism $A_\infty$ class  can be produced by infinitely  many distinct  sets of higher multiplications.

In this paper, we show how and when it is possible to detect and recover higher Massey products in the cohomology $H$ of a DGA $A$ with higher multiplications of such an $A_\infty$ structure on $H$. In what follows $\langle x_1,\dots,x_n\rangle$ and $m_n(x_1,\dots,x_n)$ denote, respectively,  the Massey product set of the cohomology classes $x_1,\dots,x_n$  and the $n$th multiplication of these classes induced by the chosen $A_\infty$ structure on $H$ which is always assumed to be quasi-isomorphic to $A$.

Given an element of a higher Massey product set there is always a particular $A_\infty$ structure on $H$ which recovers it, although in general, it does so up  to lower multiplications:

\begin{thm}\label{generalintro}{\em (Theorem {\ref{General}})}
 Let $x\in \left\langle x_1,{\dots},x_n\right\rangle$, with $n\ge 3$.

\begin{itemize}

\item[(i)] There is an $A_\infty$ structure on $H$ which recovers $x$, i.e., $\pm m_n(x_1,\dots,x_n)=x$.

\item[(ii)]  In general, for any $A_\infty$ structure on $H$, $$\pm m_n\left(x_1,{\dots},x_n\right)=x + \Gamma, \qquad \Gamma\in \sum_{j=1}^{n-1}\Img\left(m_j\right).$$
    \end{itemize}
 \end{thm}

Often, $A_\infty$ structures on $H$ are obtained by exhibiting $H$ as a contraction of $A$ via the homotopy transfer theorem. Even in this case, a particular Massey product is recovered by the $A_\infty$ structure only when the contraction is {\em adapted} (see Section \S2 for a precise definition) to the given product.

\begin{thm} \label{adaptadointro}{\em (Theorem {\ref{AdaptadoSaleMassey}})} Let $x\in \left\langle x_1,{\dots},x_n \right\rangle$. Then, for any contraction  adapted to $x$, $$\varepsilon\, m_n(x_1,{\dots},x_n)= x,$$ where $\varepsilon = (-1)^{1 + |x_{n-1}| + |x_{n-3}|+\cdots}$ .
\end{thm}

Examples \ref{ContraejemploMassey} and \ref{OtroEjem} show the accuracy of the above results and  corroborate the fact that not every $A_\infty$ structure on $H$ quasi-isomorphic to $A$ arises from a contraction. These examples are connected to topology via classical rational homotopy theory.

A more general problem is to detect when a higher multiplication of a given $A_\infty$ structure on $H$ produces a Massey product. This is not always the case and some assumptions are needed even when the $A_\infty$ structure arises from a contraction, see Theorem \ref{Lemita} or the following:

\begin{thm}  {\em (Theorem {\ref{Teoremilla}})} Let  $\left \langle x_1,{\dots},x_n\right\rangle\neq \emptyset$ with $n\ge 3$. If for some (and hence for any) homotopy retract of $A$ onto $H$, the induced higher multiplications $m_k=0$ for $k\le n-2$,  then $$\varepsilon\, m_n\left(x_1,{\dots},x_n\right)\in \left \langle x_1,{\dots},x_n\right\rangle,$$
with $\varepsilon$ as in Theorem \ref{generalintro}.
\end{thm}

We remark that results relating higher Massey products and higher multiplications on an $A_\infty$ algebra already appear in \cite{lu2009infinity,stasheff1970h}. Dually, in the sense of Eckmann-Hilton, results relating higher Whitehead products and higher brackets on an $L_\infty$ algebra were proved by the authors in \cite{HigherWhitehead}.

\subsection{Preliminaries} \label{Pre} In this paper graded objects are always assumed over $\Z$, with upper grading (in particular differentials raise the degree by 1 and all complexes are of cochains), and all algebraic structures are considered over a field $\K$.

Given $V$ a graded vector space define the \emph{tensor coalgebra} on $V$ as the graded coalgebra  $T(V)= \bigoplus_{n\geq 0}T^n(V) =\bigoplus_{n\geq 0}V^{\otimes n}$, where $1\in T^0(V)=\K$ is the counit, $\Delta (1)=1\otimes 1$, and $\Delta(v_1\otimes \cdots \otimes v_n)=\sum_{i=0}^n (v_1\otimes \cdots \otimes v_{i}) \otimes (v_{i+1}\otimes \cdots \otimes v_n).$

Recall that an \emph{A$_\infty$ algebra} is a graded vector space $A=\{A^n\}_{n\in \Z}$ together with linear maps $m_k:A^{\otimes k}\to A$ of degree $2-k$, for $k\geq 1$, satisfying the \emph{Stasheff identities} for every $i\geq1$:
 \begin{equation}\label{highermu}
\sum_{k=1}^i\, \sum_{n=0}^{i-k}(-1)^{k+n+kn}m_{i-k+1}({\rm id}^{\otimes n}\otimes m_k\otimes{\rm id}^{\otimes i-k-n})
=0.
\end{equation}
A {\em  differential graded algebra} (DGA), is an $A_\infty$ algebra for which $m_k=0$ for all $k\ge 3$. An $A_\infty$ algebra is \emph{minimal} if $m_1=0$.
Observe that $A_\infty$ algebra structures on   $A$ are in bijective correspondence with codifferentials on the tensor coalgebra $ T\left(sA\right)$ on the suspension of $A$, $(sA)^n=A^{n-1}$.
Indeed, a codifferential $\delta$ on $T\left(sA\right)$ is determined by a degree $1$ linear map $T^+\left(sA\right)\to sA$ which is written as the sum of linear maps $g_k:T^k\left(sA\right)\to sA$, $k\geq 1.$ In fact, $\delta$ is written as the sum of coderivations,
\begin{equation}\label{deltak}
\delta=\sum_{k\ge 1}\delta_k,\quad \delta_k\colon T\left(sA\right)\to T\left(sA\right),
\end{equation}
each of which being the extension as a coderivation of the corresponding $g_k$,
\begin{equation}\label{hk}
\delta_k\left(sa_1\otimes ... \otimes sa_p\right)=\sum_{i=1}^{p-k}\pm\, sa_1\otimes ... \otimes sa_{i-1}\otimes g_k\left(sa_i\otimes ... \otimes sa_{i+k-1}\right)\otimes sa_{i+k}\otimes ... \otimes sa_p.
\end{equation}
Observe, that each $\delta_k$ decreases the word length by $k-1$, that is, $\delta_k T^p\left(sA\right) \subset T^{p-k+1}\left(sA\right)$ for any $p$.

Then, the operators $\left\{m_k\right\}_{k\ge 1}$ on $A$ and the maps $\{g_k\}_{k\ge 1}$ (and hence $\delta$) uniquely determine  each other as follows:
\begin{equation}\label{relacion}
\begin{split}
m_k 	&=		s^{-1}\circ g_k\circ s^{\otimes k}\colon  A^{\otimes^k}\to A,\\[0.2cm]
g_k		&=		(-1)^{\frac{k(k-1)}{2}}s\circ m_k\circ\left(s^{-1}\right)^{\otimes k}\colon T^k\left(sA\right)\to sA.
\end{split}
\end{equation}
Note  that if $A$ is a DGA, then the corresponding codifferential $\delta$ on $T\left(sA\right)$  has only linear and quadratic part, $\delta=\delta_1+\delta_2$, determined by
 $$
g_1sa=-sda,\quad g_2(sa_1\otimes sa_2)=-(-1)^{|a_1|}s(a_1a_2).
 $$
In other words, $(T(sA),\delta)$ is the {\em bar construction of $A$}.

An \emph{$A_\infty$ morphism} $f\colon A\to A'$ between two $A_\infty$ algebras is a family of linear maps (components) $f_k\colon A^{\otimes k}\to A'$ of degree $1-k$ such that the following equation holds for every $i\geq1$:
\begin{eqnarray}\label{highermap}
\sum_{\substack{i=r+s+t \\ s\geq 1 \\ r,t\geq 0}} (-1)^{r+st}f_{r+1+t} \left(\id^{\otimes r}\otimes m_s\otimes \id^{\otimes t}\right)= \sum_{\substack{1\leq r \leq i \\ i=i_1+\cdots+i_r}} (-1)^s m_r \left(f_{i_1}\otimes \cdots \otimes f_{i_r}\right)
\end{eqnarray} being $s=\sum_{\ell=1}^{r-1}\ell(i_{r-\ell}-1).$ It is said to be an \emph{$A_\infty$ quasi-isomorphism} if $f_1\colon(A,m_1)\to (A',m_1')$ is a quasi-isomorphism of cochain complexes. Observe that
$A_\infty$ morphisms from $A$ to $A'$ are in bijective correspondence with differential graded coalgebra (DGC) morphisms $\left(T\left(sA\right),\delta\right)\to \left(T\left(sA'\right),\delta'\right)$ being $\delta$ and $\delta'$ the codifferentials determining the $A_\infty$ algebra structures.
Indeed, a DGC morphism
$$f\colon \left(T\left(sA\right),\delta\right)\longrightarrow\left(T\left(sA'\right),\delta'\right)
$$
is determined by  $\pi
f\colon T\left(sA\right)\to sA'$ ($\pi$ denotes the projection onto the indecomposables) which can be written as
$ \sum_{k\ge 1}(\pi f)_k$, where $(\pi f)_k\colon T^k\left(sA\right)\to sA'$. Note that the collection of linear maps $\{(\pi f)_k\}_{k\ge 1}$ is in one-to-one correspondence with a system $\left\{f_k\right\}_{k\ge1}$ of
linear maps $f_k\colon  A^{\otimes^k}\to A'$ of degree $1-k$ defining an $A_\infty$ morphism. Each $f_k$ and $(\pi f)_k$ determines the other by:
$$
\begin{aligned}
f_k&=s^{-1}\circ (\pi f)_k\circ s^{\otimes k},\\
(\pi f)_k&=(-1)^{\frac{k(k-1)}{2}}s\circ f_k\circ \left(s^{-1}\right)^{\otimes k}.\\
\end{aligned}
$$

In \cite[Theorem 1]{Kadeishvili}, T. Kadeishvili proved:

\begin{theorem}\label{kadei}
Given a DGA $A$ there is a minimal $A_\infty $ structure on its cohomology $H$, unique up to $A_\infty$ isomorphism, for which $H$ and $A$ are quasi-isomorphic $A_\infty$ algebras.
\end{theorem}
\begin{proof}
Choose any cochain map inclusion $j\colon (H,0)\stackrel{\simeq}{\hookrightarrow}(A,d)$ which is necessarily a quasi-isomorphism and set $j_1=j$, $m_1=0$. Assume  that $j_k$ and $m_k$ are defined for $k<p$ satisfying equations (\ref{highermu}) and (\ref{highermap}). In the latter, and for $i=p$, observe that $j_1m_p-dj_p$ is a map $U_p\colon H^{\otimes p}\to A$ involving only $\{m_k\}_{k<p}$ and $\{j_k\}_{k<p}$, whose image lies in the cycles of $A$. Define $m_p$ to be the projection of $U_p$  onto $H$ and
$j_p$ so that $dj_p=j_1m_n-U_p$. Then, $j=\{j_k\}_{k\ge1}\colon (H,\{m_k\}_{k\ge 2})\stackrel{\simeq}{\longrightarrow}$ in an $A_\infty$ quasi-isomorphism.A
\end{proof}

Observe that any $A_\infty$ structure on $H$ quasi-isomorphic to $A$ arises inductively in this way. There is however a particular procedure of obtaining such a structure whenever one exhibits $H$ as a ``homotopy retract'' of $A$:\\

A \emph{contraction} (of $M$ onto $N$) is  a diagram of the form
\begin{center}
\begin{tikzcd}
M \ar[loop left]{l}{K} \ar[shift left]{r}{q} & N, \ar[shift left]{l}{i}
\end{tikzcd}
\end{center} where $M$ and $N$ are cochain complexes and $q$ and $i$ are cochain maps such that $qi=\id_N$ and $iq\simeq \id_M$ via a chain homotopy $K$ which satisfies $K^2=Ki=qK=0$. We often denote it simply by  $(M,N,i,q,K)$. \\

We will be using the following particular statement of the  \emph{homotopy transfer theorem}, see for instance \cite{kontsevich2000deformations,Merkulov}, closely related to the \emph{homological perturbation lemma}, see for instance \cite{gugenheim1986perturbations,huebschmann1991small}:

\begin{theorem} \label{HTT} Given a contraction $(A,H,i,q,K)$  of the DGA $A$ onto its cohomology $H$ there exists a minimal $A_\infty$ algebra structure $\left\{m_k\right\}_{k\ge 2}$ on $H$, and an $A_\infty$ algebra quasi-isomorphism $A\stackrel{I}{\longleftarrow}H$ extending $i$.\hfill$\square$
\end{theorem}
In this case, the transferred higher products $\left\{m_k\right\}$ and the components $\left\{I_k\right\}$ of $I$ are  given inductively as follows: formally, set $K\lambda_1=-i$, and define $\lambda_k\colon H^{\otimes n}\to A$, $k\geq 2$, recursively by
\begin{equation}\label{FormulaLambda}
\lambda_k=m\left(\sum_{s=1}^{k-1} (-1)^{s+1} K\lambda_s\otimes K\lambda_{k-s}\right),
\end{equation}
where $m$ denotes the multiplication on $A$. Then,  	
\begin{equation*}
m_k =q\circ\lambda_k \qquad \textrm{ and } \qquad I_k = K\circ\lambda_k \qquad \textrm{for all } k\geq 2.
\end{equation*}

There are equivalent descriptions of  $\left\{m_k\right\}$ and $\left\{I_k\right\}$ in terms of labeled binary trees but the  proofs in the results to follow are not shorter using this approach.

\begin{remark}\label{ExisteRetracto}\rm Note that contractions $(A,H,i,q,K)$ are in bijective correspondence with decompositions of $A$ of the form
$$
A=B\oplus dB\oplus C,
$$
where $B$ is a complement of $\Ker d$ (and thus $d\colon B\stackrel{\cong}{\to} dB$) and $C\cong H$.
Indeed, for such a decomposition define $i\colon H\cong C\hookrightarrow A$, $q\colon A\twoheadrightarrow C\cong H$ and  $K(B)=K(C)=0$, $K\colon d B \stackrel{\cong}{\to} B$.
Conversely, given a contraction
 $(A,H,i,q,K)$   define $
B=KdA$ and $C=\Img i$.
 \end{remark}

  In what follows,  write  $\overline a = (-1)^{|a|+1}a$
 where $a$ is a homogeneous element of $A$ and $|a|$ denotes its degree. Let $x_1,x_2,x_3\in H$ be cohomology classes of a given DGA $A$ such that $x_1x_2=x_2x_3=0$. A \emph{defining system} (for the triple Massey product) is a set $\{a_{ij}\}_{0\leq i<j\leq 3, 1\leq j-i\leq 2} \subseteq A$ defined as follows:

 For $i=1,2,3$ choose a cocycle $a_{i-1,i}$ representative of $x_i$. These are $\{a_{01},a_{12},a_{23}\}$.
	
	For $0 \leq i < j \leq 3$ and $j-i=2$, choose $a_{ij}\in A$ with the property that $d(a_{ij})=\overline{a}_{i,i+1}a_{i+1,j}$. These are $\{a_{02},a_{13}\}$.

Recall that the \emph{triple Massey product}  is defined as the set (empty if the condition $x_1x_2=x_2x_3=0$ is not satisfied):
 $$\langle x_1,x_2,x_3 \rangle=\left\{[\overline{a}_{01}a_{13} + \overline{a}_{02}a_{23}],\, \{a_{ij}\} \textrm{ is a defining system}\right\}\subseteq H^{s-1},$$
 with $ s=|x_1|+|x_2|+|x_3|$.

It is also classical to define $$\operatorname{In}(x_1,x_2,x_3)=x_1H^{|x_2|+|x_3|-1}+H^{|x_1|+|x_2|-1}x_3$$ as the \emph{indeterminacy subgroup}, and regard the triple Massey products set as an element in the quotient $H^{s-1}/\operatorname{In}(x_1,x_2,x_3)$.
However, this does not simplify the arguments in the paper as higher products on a given $A_\infty$ structure are not determined ``up to indeterminacies''.

Higher Massey products are inductively defined. Let $x_1,{\dots},x_n\in H$ be such that for $1\leq i<j\leq n$ and $ j-i\leq n-2$, $\langle x_i,{\dots},x_j \rangle$ is trivial, that is, it contains the zero class.  A \emph{defining system} (for the $n$th order Massey product) is a set $\{a_{ij}\}_{0\le i<j\le n, 1\leq j-i \leq n-1} \subseteq A$ defined as follows.
\begin{itemize}
	\item For $i=1,{\dots},n$ choose a cocycle $a_{i-1,i}$ representative of $x_i$.
	\item For $0 \leq i < j \leq n$ and $2\leq j-i\leq n-1$, choose $a_{ij}\in A$ with the property that $$d(a_{ij})=\sum_{0\leq i<k<j\leq n} \overline{a}_{ik}a_{kj} .$$ Their existence follows from the triviality of $\langle x_i,{\dots},x_j \rangle$.
\end{itemize} The \emph{$n$th order Massey product} is then the set $$\langle x_1,{\dots},x_n \rangle=\Bigl\{\Bigl[\sum_{0\leq i<k<j\leq n} \overline{a}_{ik}a_{kj}\Bigr],\, \{a_{ij}\} \textrm{ is a defining system}\Bigr\}\subseteq H^{s+2-n}$$
 where $s=\sum_{i=1}^n |x_i|$. If some $\langle x_i,{\dots},x_j \rangle$ is not trivial define $\langle x_1,{\dots},x_n \rangle$ as the empty set.\\

In what follows any $A_\infty$ structure considered in the cohomology $H$ of a given DGA $A$ is always assumed to be quasi-isomorphic to $A$.

\begin{definition}\label{detection}{\em The $n$th multiplication $m_n$ of a given $A_\infty$ structure  is said to \emph{recover the Massey product element $x\in \left\langle x_1,{\dots},x_n\right\rangle$} if, up to sign,  $ m_n(x_1,{\dots},x_n)= x$.

 More generally, we say that  $m_n$ {\em detects the Massey product set $\left\langle x_1,{\dots},x_n\right\rangle$}  if, up to sign,  $ m_n(x_1,{\dots},x_n)\in \left\langle x_1,{\dots},x_n\right\rangle$.}
 \end{definition}

\section{Recovering Massey products} \label{MasProd}

We begin by showing that given a higher Massey product on the cohomology $H$ of a given DGA $A$, there is always a particular $A_\infty$ structure on $H$ which recovers it. In general, an arbitrary $A_\infty$ structure only recovers the given Massey product up to multiplications of lower arity.

\begin{theorem}  \label{General} Let $x\in \left\langle x_1,{\dots},x_n\right\rangle$, with $n\ge 3$.

\begin{itemize}

\item[(i)] There is an $A_\infty$ structure on $H$ which recovers $x$.

\item[(ii)]  In general, for any $A_\infty$ structure on $H$, $$\varepsilon\, m_n\left(x_1,{\dots},x_n\right)=x + \Gamma, \qquad \Gamma\in \sum_{j=1}^{n-1}\Img\left(m_j\right),\quad \varepsilon=(-1)^{\sum_{j=1}^{n-1}(n-j)|x_j|}.$$
    \end{itemize}
\end{theorem}

\begin{proof} (i) We construct the $A_\infty$ structure on $H$ recovering $x$  following the algorithm in Theorem \ref{kadei}. For it note that  the map $U_p$ in the proof of that result  is explicitly given by,
\begin{align*}
U_p(y_1,...,y_p) &= \sum_{s=1}^{p-1} \varepsilon(y_1,...,y_s) \, j_s\left(y_1,...,y_s\right)j_{p-s}\left(y_{s+1},...,y_p\right) \\
&\quad + \sum_{r=2}^{p-1} \sum_{\ell=0}^{p-r}  \eta(y_1,...,y_\ell) \  j_{p-r+1}\left(y_1,...,m_r\left(y_{\ell+1},...,y_{\ell+r}\right),...,y_p\right),
\end{align*}
where $y_1,\dots,y_p\in H$, $\varepsilon(y_1,...,y_s)$  is the parity of $s+\left(p-s+1\right)\left(|y_1|+\cdots +|y_s|\right)$, and $\eta(y_1,...,y_\ell)$ is the parity of $(\ell+r)(p-\ell-r+|y_1|+\cdots+|y_\ell|)$.

 We proceed by induction on $n$. For $n=3$ choose a defining system $\left\{ a_{ij} \right\}$ for the Massey product $x\in \langle x_1,x_2,x_3\rangle$ and define $j_1\colon H\to A$ any quasi-isomorphism extending the choice $j_1(x_1)=a_{01}$, $j_1(x_2)=a_{12}$ and $j_1(x_3)=a_{23}$. As $x_1x_2=x_2x_3=0$ we have:
\begin{equation*}
U_2(x_1,x_2) = - j_1(x_1)j_1(x_2)=-a_{01}a_{12}=(-1)^{|x_1|} \overline a_{01}a_{12} = (-1)^{|x_1|}da_{02}.
\end{equation*} Analogously, $U_2(x_2,x_3)=(-1)^{|x_2|} da_{13}.$ Define $j_2\colon H^{\otimes 2} \to A$ by

\begin{equation*}
j_2(x_1,x_2) = (-1)^{|x_1|} a_{02} \qquad \textrm{and } \qquad j_2(x_2,x_3)=(-1)^{|x_2|} a_{13}.
\end{equation*} Then,

\begin{align*}
U_3(x_1,x_2,x_3) &= -(-1)^{|x_1|} j_1(x_1)j_2(x_2,x_3) + j_2(x_1,x_2)j_1(x_3)\\
&= -(-1)^{|x_1|+|x_2|} a_{01}a_{13} + (-1)^{|x_1|} a_{02}a_{23} = (-1)^{|x_2|} \left(\overline a_{01}a_{13} + \overline a_{02}a_{23} \right).
\end{align*} Hence, $$ m_3(x_1,x_2,x_3)=\left[ U_3\left(x_1,x_2,x_3\right)\right] = (-1)^{|x_2|}x .$$
Any choice of $\{m_k\}_{k>3}$ and $\{j_k\}_{k>3}$ complete the searched $A_\infty$ structure on $H$.

  Next, let $\left\{ a_{ij} \right\}$ be a defining system for  $x\in \langle x_1,...,x_n\rangle$ and assume as induction hypothesis that we have constructed $\{m_k\}_{k<n}$ and $\{j_k\}_{k<n}$ so that:
$$
m_k\left(x_{i+1},...,x_{i+k}\right) = 0 \in \left\langle  x_{i+1},...,x_{i+k} \right\rangle,\quad 2\le k<n, \quad 0 \leq i \leq n-k,
$$
$$
 j_k\left( x_{i+1},...,x_{i+k} \right)=\varepsilon_k\,  a_{i,i+k},\quad 1 \leq k <n, \quad 0 \leq i \leq n-k,
$$

with $ \varepsilon_k = (-1)^{1 + |x_{i+k-1}| + |x_{i+k-3}|+{\cdots}}$ .
 Then,
\begin{align*}
U_n(x_1,...,x_n) &= \sum_{s=1}^{n-1} \varepsilon(x_1,...,x_s) \, j_s\left(x_1,...,x_s\right)j_{n-s}\left(x_{s+1},...,x_n\right)\\
&= \sum_{s=1}^{n-1} \varepsilon_s \varepsilon(x_1,...,x_s)\, a_{0,s}a_{s,n} = \varepsilon\, \sum_{s=1}^{n-1} \overline{a}_{0,s}a_{s,n}
\end{align*}

with $\varepsilon = (-1)^{1 + |x_{n-1}| + |x_{n-3}|+{\cdots}}$. Hence, $$\varepsilon m_n(x_1,...,x_n)=\varepsilon\,\left[U_n(x_1,...,x_n)\right]=\varepsilon \,x.$$
Again, finish by completing the $A_\infty$ structure on $H$ with any choice of $\{m_k\}_{k>n}$ and $\{j_k\}_{k>n}$.

(ii) This is the Eckmann-Hilton dual of \cite[Prop. 3.1]{HigherWhitehead}. Recall (see for instance \cite[III.6]{tanre1984homotopie}) that the Eilenberg-Moore  spectral sequence of $A$ is the coalgebra spectral sequence obtained by filtering the bar construction $(T(sA),\delta)$ by the ascending filtration  $F_p=\left(sA\right)^{\otimes \leq p}$. Consider any $A_\infty$ quasi-isomorphism $ H\stackrel{\simeq}{\longrightarrow} A$ given by Theorem \ref{kadei} and the corresponding  DGC quasi-isomorphism,
$$
 (T(sH),\delta)\stackrel{\simeq}{\longrightarrow} (T(sA),\delta).
$$
Choosing the same filtration on $T(sH)$ we observe that at the $E^1$ level the induced morphism of spectral sequences is the identity on $T(sH)$. By comparison, all the terms in both spectral sequences are also isomorphic. Now, translating \cite[Thm. V.7(6)]{tanre1984homotopie}  to the spectral sequence on $T(sH)$ we obtain that if $\langle x_1,\dots,x_n\rangle$ is non empty, then the element $sx_1\otimes\dots\otimes sx_n$ survives  to  the $n-1$ page $(E^{n-1},\delta^{n-1})$. Moreover, given any $x\in \langle x_1,\ldots,x_n\rangle $, one has
$$\delta^{n-1}\,\,\overline{sx_1\otimes\dots\otimes sx_n}=\overline{sx}.
$$
Here $\overline{\,\,\cdot\,\,}$ denotes the class in $E^{n-1}$. In other words, there exists $\Phi\in T^{\le n-1}(sH)$ such that
\begin{equation}\label{ole}
\delta\left(sx_1\otimes\dots\otimes sx_n+\Phi\right)=s x.
\end{equation}
Write $\delta=\sum_{i\ge 2}\delta_i$ with each $\delta_i$ as in equation $(\ref{deltak})$,  and decompose $\Phi=\sum_{i=2}^{n-1}\Phi_i$ with $\Phi_i\in T^i(sH)$. By a word length argument,
$$
 \delta_k(sx_1\otimes\dots\otimes sx_n)+\sum_{i=2}^{n-1}\delta_i(\Phi_i)=sx.
 $$
 Note also that $\delta_p=g_p$ for elements of word length $p$, with $g_p$ as in equation (\ref{hk}). Therefore, $$g_n\left(sx_1\otimes\dots\otimes sx_n\right)+\sum_{i=2}^{n-1}g_i(\Phi_i)=sx.$$
To finish, apply the identities (\ref{relacion}) and write each $g_i$ in terms of the corresponding $m_i$ for all  $i=1,\dots,n$.
In particular, the sign $\varepsilon$ appears when writing
$$
m_n(x_1,\dots, x_n)=s^{-1}\circ g_n\circ s^{\otimes n}(x_1,\dots, x_n)=\varepsilon\,s^{-1} g_n(sx_1\otimes\dots\otimes sx_n).
$$
\end{proof}

\begin{corollary} \label{CorAGeneral} Let $A$ be a DGA such that for some (and hence for any) $A_\infty$ structure on $H$,  $m_k=0$ for $1\le k\le n-1$. Then,   for any cohomology classes $x_1,{\dots},x_k\in H$,  the Massey product set $\left \langle x_1,\dots,x_n\right\rangle=\left\{x\right\}$ consists of a single class which is recovered by the $n$th multiplication, that is, $\varepsilon\, m_n\left(x_1,{\dots},x_n\right)=x$, with $\varepsilon$ as in Theorem \ref{General}.
\end{corollary}

\begin{proof} By induction on $s=j-i$, we have that
\begin{equation*}
\left\{0\right\} = \left\langle x_i,{\dots},x_j \right\rangle \neq \emptyset \quad \textrm{for all} \quad 1\leq i  < \cdots < j \leq n, \quad \textrm{and } j-i \leq n-1.
\end{equation*} Now given $x\in\left \langle x_1,{\dots},x_n\right\rangle,$ the result follows from a direct application of Theorem \ref{General}.
\end{proof}

\begin{remark}\label{otro}\rm Note that the least  $k$ for which the $k$th multiplication $m_k$ is non trivial does not depend on the given $A_\infty$ structure. Hence the result above holds for all of them. The same applies for Theorem \ref{Teoremilla}.
\end{remark}

 Even when the $A_\infty$ structure on $H$ is induced by a contraction, Massey products are not always recovered by higher multiplications unless restrictive hypothesis are imposed.

\begin{definition}\label{RetractoAdaptado2} \rm Let $x\in \left\langle x_1,{\dots},x_n\right\rangle$. A contraction $(A,H,i,q,K)$ is \emph{adapted to $x$} if there exists a defining system $\{a_{ij}\}$ for $x$ such that $i(x_j)=a_{j-1,j}$ for every $j$ and $\{a_{ij}\}_{j-i\geq 2} \subseteq KdA$.\end{definition}

Recall from Remark \ref{ExisteRetracto} that $KdA=B$ in the decomposition $A=B\oplus dB \oplus C$ associated to the contraction

\begin{theorem}\label{AdaptadoSaleMassey}  Let $x\in \left\langle x_1,{\dots},x_n \right\rangle$. Then, for any contraction adapted to $x$, $$\varepsilon\, m_n(x_1,{\dots},x_n)= x,$$ where $\varepsilon = (-1)^{1 + |x_{n-1}| + |x_{n-3}|+{\cdots}}$ .
\end{theorem}

\begin{proof}Let $x\in \langle x_1,{\dots},x_n\rangle$ and let $\{a_{ij}\}$ be a defining system for which the contraction $(A,H,i,q,K)$ is adapted to $x$. Consider the map
$
\lambda_n\colon H^{\otimes n}\longrightarrow A,\quad n\ge 2,
$
in formula (\ref{FormulaLambda}) for which $
m_n=q\circ\lambda_n$.
First, we prove by induction on $s$, for $2\leq s \leq n-1$, that,
\begin{equation}\label{Paso1}
K\lambda_s(x_1,{\dots},x_s)=(-1)^{b_s}a_{0s},
\end{equation}
where $b_s = |x_{s-1}| + |x_{s-3}| + \cdots + 1$.
For $s=2$, it is straightforward. Assume that equation (\ref{Paso1}) holds for every $p\leq s-1,$ and prove it for $p=s$:
\begin{align*}
&K\lambda_s(x_1,{\dots},x_s)=K\lambda_2\left( \sum_{i=1}^{s-1} (-1)^{i+1} K\lambda_i\otimes K\lambda_{s-i} \right) (x_1,{\dots},x_s)\\
&=K\lambda_2\left( \sum_{i=1}^{s-1} (-1)^{i+1+(|x_1|+\cdots+|x_i|)(i-s+1)} K\lambda_i (x_1,{\dots},x_i)\otimes K\lambda_{s-i} (x_{i+1},\dots,x_s)\right)\\
\intertext{apply the induction hypothesis and note that, recursively, $$1+|a_{0k}|=|x_1|+\cdots+|x_k| - (k-2) \qquad \textrm{for every $k=1,...,i.$}$$  Then,}
&=K\lambda_2 \left( \sum_{i=1}^{s-1} (-1)^{i+1+(|x_1|+\cdots+|x_i|)(i-s+1)+b_i+b^{s-i}} a_{0i}\otimes a_{is}\right)\\
&=K\lambda_2 \left( \sum_{i=1}^{s-1} (-1)^{i+1+(|x_1|+\cdots+|x_i|)(i-s+1)+b_i+b^{s-i}+1+|a_{0k}|} \bar a_{0i}\otimes a_{is}\right) \\
&=  (-1)^{b_s}K\left( \sum_{0<i<s} \bar a_{0i}a_{is}\right) = (-1)^{b_s}Kd(a_{0s})=(-1)^{b_s}a_{0s}.
\end{align*}
An analogous argument proves that
\begin{equation}\label{Paso2}
K\lambda_{n-s}(x_{s+1},{\dots},x_n)=(-1)^{b^{n-s}}a_{sn},
\end{equation}
where $b^{n-s} = |x_{n-1}|+|x_{n-3}|+\cdots + 1$.  Using equations (\ref{Paso1}) and (\ref{Paso2}) we have,

\begin{align*}
&\lambda_n(x_1,{\dots},x_n)=\lambda_2\left( \sum_{s=1}^{n-1}(-1)^{s+1}K\lambda_s\otimes K\lambda_{n-s} \right) (x_1,{\dots},x_n)  \\
&= \lambda_2 \left( \sum_{s=1}^{n-1} (-1)^{s+1+\left(\sum_{i=1}^s |x_i|\right)(s-n+1)} K\lambda_s(x_1,{\dots},x_s)\otimes k\lambda_{n-s}(x_{s+1},{\dots},x_n)\right)\\
&= \sum_{s=1}^{n-1} (-1)^b \bar a_{0s}a_{sn},
\end{align*}where $b=|x_{n-1}|+|x_{n-3}|+\cdots+ 1$. Hence, $m_n(x_1,\dots,x_n)=\varepsilon\, x$ \end{proof}

The next example shows that if a contraction is not adapted to a given Massey product, then the associated higher multiplication might not recover it.

\begin{example} \label{ContraejemploMassey} \rm Let $(\Lambda V,d)$ be the commutative DGA over $\Q$, where  \begin{equation*}
V=\Span\{\underbrace{a_{01},a_{12},a_{23},a_{34}}_{\textrm{degree 3}},\quad  \underbrace{a_{02},a_{13},a_{24},z_1,z_2}_{\textrm{degree } 5}, \quad \underbrace{a_{03},a_{14}}_{\textrm{degree } 7}\},
\end{equation*}$a_{i-1,i}$ and $z_1,z_2$ are cocycles, and for the rest of the elements $\displaystyle d a_{ij}=\sum_{i<k<j} a_{ik}a_{kj}.$ \\
It is straightforward to check that, fixing the cohomology classes $x_i=[a_{i-1,i}]$ for $i=1,2,3,4$, there is a unique defining system $\{a_{ij}\}$ (given by the obvious choices) which gives rise to a unique non trivial Massey product $x=[a_{01}a_{14}+a_{02}a_{24}+a_{03}a_{34}]$. That is, $\langle x_1,x_2,x_3,x_4\rangle=\{x\}\subseteq H^{10}$ and $x\neq 0$.
It is also easy to see that the cohomology in degree $10$ admits the basis $\left\{ x,[z_1z_2]\right\}$.

Next, fix the decomposition of $(\Lambda V,d)$ as $B\oplus dB\oplus C$  like in Remark \ref{ExisteRetracto} given in the table below. On it, elements appearing in $(dB)^s$ come from differentiating the elements of $B^{s-1}$ in the order written, and a dot $\cdot$ indicates that the corresponding subspace is the trivial one. The decomposition above degree $10$ is irrelevant for our purposes.
\begin{center}
\begin{tabular}{cccc}
degree & $B$ & $dB$ & $C $\\ \toprule
3& $\cdot$ & $\cdot$ & $a_{01},a_{12},a_{23},a_{34}$\\ \midrule
4& $\cdot$ & $\cdot$& $\cdot$\\ \midrule
5& $\begin{array}{c}a_{02}+z_1, a_{13},\\ a_{24}+z_2\end{array}$& $\cdot$ & $z_1,z_2$\\ \midrule
6&$\cdot$ & $\begin{array}{c}a_{01}a_{12}, a_{12}a_{23},\\ a_{23}a_{34}\end{array}$& $a_{01}a_{23}, a_{01}a_{34},a_{12}a_{34}$\\ \midrule
7& $a_{03}, a_{14}$& $\cdot$ &  $\cdot$ \\ \midrule
8&$\begin{array}{c}a_{01}a_{13},a_{01}a_{24},\\a_{02}a_{34},a_{12}a_{24}\end{array}$  & $\begin{array}{c}a_{01}a_{13}+a_{02}a_{23},\\ a_{12}a_{24}+a_{13}a_{34}\end{array}$ &  $\begin{array}{c} a_{01}a_{23},a_{02}a_{12},a_{12}a_{13}, \\ a_{13}a_{23},a_{23}a_{24},a_{23}a_{34}, \\ a_{01}z_1,a_{12}z_1,a_{23}z_1,a_{34}z_1, \\ a_{01}z_2,a_{12}z_2,a_{23}z_2,a_{34}z_2\end{array}$\\ \midrule
9& $\cdot$ & $\begin{array}{c}a_{01}a_{13}a_{23},a_{01}a_{23}a_{34},\\ a_{01}a_{12}a_{34},a_{12}a_{23}a_{34}\end{array}$ &  $\cdot$ \\ \midrule
10& & $\cdot$ & $\begin{array}{c}a_{01}a_{14}+a_{02}a_{24}+a_{03}a_{34}\\ z_1z_2\end{array}$\\ \midrule
\end{tabular}\\
\end{center}

Then, a straightforward computation for the contraction associated to this decomposition and the induced $A_\infty$ structure on $H^*(\Lambda V,d)$ gives the cohomology class
 $$m_4(x_1,x_2,x_3,x_4)=-x-[z_1z_2],$$ which is \emph{not} a Massey product by the discussion above.

\end{example}

\begin{remark}\rm By basic facts on rational homotopy theory \cite{felix2012rational}, the commutative DGA $(\Lambda V,d)$ in the example above is the minimal model of  a simply connected elliptic complex $X$ which is of the form
$$
X=S^5\times S^5\times Y$$
where $Y$ lies as the total space of a fibration of the sort
$$
(S^5)^{\times 3}\times (S^7)^{\times2}\longrightarrow Y\longrightarrow (S^3)^{\times 4}.
$$
Hence, Example \ref{ContraejemploMassey}  shows that in $H^*(X;\Q)$ the set $\langle x_1,x_2,x_3,x_4\rangle$ reduces to the single element $x$ which is not recovered by the $A_\infty$ structure induced by the given decomposition. \end{remark}

The following is an example of a DGA with infinitely many triple Massey products that are never recovered by any $A_\infty$ structure on its cohomology induced by a contraction. In particular, there are infinitely many distinct(although isomorphic)  $A_\infty$ structures which do not arise by contractions.

\begin{example}\label{OtroEjem}\rm Let $(\Lambda V,d)$ be the commutative DGA over $\Q$ where $$V=\Span\{\underbrace{a_{01},a_{12},a_{23}}_{\textrm{degree 3}}, \underbrace{a_{02},a_{13}}_{\textrm{degree } 5}\}$$ and $$d{a_{01}}=d{ a_{12}}=d{a_{23}}=0, \quad  d{a_{02}}=a_{01}a_{12},\quad  d{a_{13}}=a_{12}a_{23}.$$ Let
$$
A=(\Lambda V,d)/J
$$
where  $J$ is the differential ideal generated by $\{a_{01}a_{12},a_{12}a_{23}\}$. We denote the elements of the quotient algebra as in the original without confusion.

Fixed $x_1=[a_{01}], x_2=[a_{12}]$ and $x_3=[a_{23}]$, the possible defining systems $\{b_{ij}\}$ for the triple Massey product $\langle x_1,x_2,x_3\rangle$ are of the following form, where $\alpha_k,\beta_k\in \Q$:

$$b_{01}=a_{01}, \quad b_{12}=a_{12},  \quad b_{23}=a_{23},$$
$$b_{02}=\alpha_1 a_{02}+\alpha_2a_{13}\quad \textrm{ and }   \quad  b_{13}=\beta_1 a_{02}+\beta_2a_{13}.$$

This implies that the triple Massey product set is $$\langle x_1,x_2,x_3\rangle=\big\{\alpha_1[a_{01}a_{02}]+\alpha_2[a_{01}a_{13}]+\beta_1[a_{02}a_{23}]+\beta_2[a_{13}a_{23}] \mid \alpha_k,\beta_k \in \Q\big\}.$$ The zero class belongs to the set, but there are infinitely many other non trivial Massey product elements. It is straightforward  to check that for any contraction, one has that $m_3(x_1,x_2,x_3)=0$. Therefore, one never recovers a non trivial Massey product element. \hfill$\square$ \\
\end{example}

\section{Detecting Massey products}\label{DiscusionLu}

In Theorem 3.1 of the interesting paper \cite{lu2009infinity}, the authors prove that higher multiplications of $A_\infty$ structures arising from contractions always detect Massey products: for any cohomology classes $x_1,\dots, x_n$ in the cohomology $H$ of a given DGA $A$ such that $\langle x_1,\dots,x_n\rangle$ is non empty, and for any $A_\infty$ structure on $H$ induced by a contraction,
$$
\varepsilon \, m_n(x_1,\dots,x_n)\in \langle x_1,\dots,x_n\rangle,
$$
where $\varepsilon$ is as in Theorem \ref{AdaptadoSaleMassey}.

Unfortunately, as stated, this result is only valid for $n=3$ which is the first inductive step in its proof, and also, our Corollary \ref{Caso31} below. For $n\ge 4$, Example \ref{ContraejemploMassey} is a clear counterexample. The small gap in the proof occurs when assuming that  the elements
$$\left\{ a_{ij}= K\lambda_{j-i+1}(x_i,{\dots},x_j) \mid 2 < j-i < n-1\right\}$$
together with suitable representatives $a_{i-1,i}$ of the classes $x_i$ form a defining system. This is not the case, for instance, in the contraction chosen in Example \ref{ContraejemploMassey}.
If one imposes  this extra assumption, then the inductive proof in \cite[Thm. 3.1]{lu2009infinity} works and it shows the following:

\begin{theorem}\label{Lemita} Let $A$ be a DGA and assume $\langle x_1,{\dots},x_n \rangle \neq \emptyset$, $n\ge 3$. Then, for any contraction of $A$ such that the elements $$\left\{ a_{ij}= K\lambda_{j-i+1}(x_i,{\dots},x_j) \mid 2 < j-i < n-1\right\}$$  assemble into a defining system,  $$\varepsilon\,  m_n(x_1,{\dots},x_n)\in \langle x_1,{\dots},x_n \rangle$$ with $\varepsilon$ as in Theorem \ref{AdaptadoSaleMassey}.\hfill$\square$
\end{theorem}

To ensure that any $A_\infty$ structure on $H$ arising from a contraction detects Massey product, extra assumptions are needed:

\begin{theorem} \label{Teoremilla}  Let  $\left \langle x_1,{\dots},x_n\right\rangle\neq \emptyset$ with $n\ge 3$. If for some (and hence for any) contraction of $A$ onto $H$, the induced higher multiplications $m_k=0$ for $k\le n-2$,  then $$\varepsilon\, m_n\left(x_1,{\dots},x_n\right)\in \left \langle x_1,{\dots},x_n\right\rangle,$$
 with $\varepsilon$ as in Theorem \ref{General}.
\end{theorem}

\begin{proof} Recall that $\left \langle x_1,{\dots},x_k\right\rangle\neq \emptyset$ implies that
$
0 \in \left \langle x_i,{\dots},x_j\right\rangle$, for any $3 \leq j-i \leq k-1$.
Therefore, we apply Theorem \ref{General}, taking into account that $m_i=0$ for $i\leq k-2$ to deduce,
\begin{equation}\label{Ecuat}
m_{k-1}\left(x_i,{\dots},x_j\right) = 0 \quad \textrm{ for any } \quad j-i = k-1.
\end{equation} Let $A=B\oplus dB \oplus C$ be the decomposition equivalent to the chosen contraction. By induction on $p$, with $2\leq p \leq k-1$, we will construct a set of elements $\left\{a_{ij}\right\}_{2 \leq j-i \leq p} \subseteq B$ with the property that $d(a_{ij})=\sum_{i<l<j}\bar a_{il}a_{lj}.$

For each $i$, we denote by $x_i'$ a cocycle representing $x_i$. Let $p=2$. As $\left \langle x_1,{\dots},x_k\right\rangle\neq \emptyset$, we can (and do) define $a_{ij}= Kd b_{ij}$, being $b_{ij}$ any election such that $d(b_{ij})=x_i'x_j'$. Then, $a_{ij}\in B$ by construction, and the differential behaves as expected. \\ Assume the assertion true for $p\leq k-2.$ Then, there exists a family of elements of $B$, $\left\{a_{ij}\right\}_{2 \leq j-i \leq k-2}$, such that $d(a_{ij})=\sum_{i<l<j}\bar a_{il}a_{lj}.$ Now, as the contraction is adapted to the defining system we are building, the same argument as in the proof of Theorem \ref{AdaptadoSaleMassey} proves that
$$m_p\left(x_i,{\dots},x_j\right)=q\left(\sum_{i<l<j} \bar a_{il}a_{lj}\right)\quad \textrm{ for any }\quad 3 \leq p=j-i\leq k-2.$$ By equation (\ref{Ecuat}), $$q\left(\sum_{i<l<j}\bar a_{il}a_{lj}\right)=0 \quad \textrm{ for } \quad j-i=k-1.$$ Hence, there exists some $\Psi$ with $d\Psi=\sum_{i<l<j}\bar a_{il}a_{lj}$. Finally, define $$a_{ij}=K\left(\sum_{i<l<j} \bar a_{il}a_{lj}\right) \quad \textrm{ for  }\quad  j-i = k-1,$$ which belongs to $B$ and satisfies our claim, proving the result.
\end{proof}

\begin{corollary} \label{Caso31} Let $A$ be a DGA with cohomology $H$.
\begin{itemize}
\item[(i)] If $\langle x_1,x_2,x_3\rangle$ is non empty, then, for any $A_\infty$ structure on $H$ induced by a contraction,
$$\varepsilon
\,m_3(x_1,x_2,x_3)\in \langle x_1,x_2,x_3\rangle.
$$
\item[(ii)] If the product on $H$ is trivial and $\langle x_1,x_2,x_3,x_4\rangle$ is non empty, then, for any $A_\infty$ structure on $H$ induced by a contraction, $$\varepsilon
\,m_4(x_1,x_2,x_3,x_4)\in \langle x_1,x_2,x_3,x_4\rangle.\eqno{\square}$$
\end{itemize}
\end{corollary}

\bibliographystyle{plain}
\bibliography{MyBib}

\noindent\sc{Departamento de \'Algebra, Geometr{\'\i}a y Topolog{\'\i}a,\\
Universidad de M\'alaga, Ap. 59, 29080 M\'alaga, Spain.}\\
\\
\noindent\tt{ubuijs@uma.es, josemoreno@uma.es, aniceto@uma.es}

\end{document}